\documentclass[10pt]{article}
\usepackage{mathrsfs}
\usepackage{amsfonts}
\usepackage{amsthm,amsmath,amssymb,anysize}
\usepackage{color}
\newtheorem{lemma}{Lemma}[section]
\newtheorem{theorem}[lemma]{Theorem}

\newtheorem{proposition}[lemma]{Proposition}
\newtheorem{definition}[lemma]{Definition}

\usepackage[numbers,sort&compress]{natbib}
\marginsize{3.5cm}{3.5cm}{3.6cm}{3cm}
\numberwithin{equation}{section}

\title{\textsf{Classification of  nilpotent Lie superalgebras of   multiplier-rank $\leq 2$}}
\author{\textsc{Wende Liu$^{1,}$}\footnote{Supported by the NSF
  of China (11471090, 11701158)} and
  \textsc{Yanling Zhang$^{2,}$}\footnote{Correspondence:  wendeliu@ustc.edu.cn (Y. Zhang)}
\\
\\
\textit{$^1$School of Mathematics and Statistics,}
\textit{Hainan Normal University,}\\\textit{ Haikou 571158, P. R. China}\\
 \textit{$^2$School of Mathematical Sciences},
  \textit{Harbin Normal University} \\
  \textit{Harbin 150025, P. R. China}
  }
\date{ }

\begin{document}
\maketitle
\begin{quotation}
\small\noindent \textbf{Abstract}:
In this paper, we introduce the notion of (super-)multiplier-ranks for Lie superalgeras and classify all the finite-dimensional nilpotent Lie superalgebras of multiplier-rank $\leq 2$ over an algebraically closed field of characteristic zero. In the process, we also determine the multipliers of Heisenberg superalgebras.

\vspace{0.2cm} \noindent{\textbf{Keywords}}: Lie superalgebra; multipliers; (super-)multiplier-rank

\vspace{0.2cm} \noindent{\textbf{Mathematics Subject Classification 2010}}: 17B05, 17B30, 17B56
\end{quotation}

\setcounter{section}{0}
\section{Introduction}

As is well known, the notion of  multipliers and covers for a group arose from  Schur's work on projective representations of groups.
 Analogous to the group theory case, for a finite-dimensional Lie algebra $L$ over a field, a cover is a central extension of the maximal possible dimension of $L$ with a kernel contained in the derived algebra of $L$ and the corresponding kernel is   a (Schur) multiplier of $L$. For a  finite-dimensional Lie algebra, there exist uniquely a cover  and  a multiplier up to Lie algebra isomorphism, respectively. A typical fact  is that the multiplier of a finite-dimensional Lie algebra $L$ is isomorphic to the second cohomology group of $L$ with coefficients in the 1-dimensional trivial module \cite{Batten}.
The study on  multipliers  of Lie algebras began in 1990's (see \cite{Batten-Stitzinger,K.Moneyhun}, for example) and  the theory  has seen a fruitful  development (see \cite{B-M-S,E-S-D,Hardy,Hardy-stitzinger,Nir,Niroomand3,SAN}, for example). Among the literatures, a main work is finding  an upper bound  for the multiplier dimension for a finite-dimensional nilpotent Lie algebra and   classifying finite-dimensional nilpotent Lie algebras under certain conditions in terms of multipliers
(see \cite{B-M-S,Hardy, Hardy-stitzinger,Nir,Niroomand3,SAN}, for example).

The notion of multipliers for Lie algebras may be naturally generalized to the Lie superalgebra case. In this paper, we first establish several lemmas for Lie superalgebras, which are parallel to the ones in non-super case. Then we introduce the notions of  (super-)multiplier-ranks and  (super-)derived-ranks, which
are analogous to the two invariants  in Lie algebra case. Our main result is classifying all the nilpotent Lie superalgebras of  multiplier-rank $\leq 2$. As a byproduct, we also determine the multipliers of Heisenberg superalgebras.

\section{Basics}
In this paper, all (linear) superspaces and superalgebras are over an algebraically closed field $\mathbb{F}$ of characteristic zero. Let
  $\mathbb{Z}_{2}:=\{\bar{0}, \bar{1}\}$ the abelian group of order $2$ and $V=V_{\bar{0}}\oplus
V_{\bar{1}}$ a superspace.  For a  homogeneous element $x$ in $V$, write $|x|$ for the
parity of $x$. The symbol $|x|$  implies that $x$ has been assumed to be  a  homogeneous element.
In $\mathbb{Z} \times \mathbb{Z}$, we define a partial order as follows:
$$(m, n)\leq(k, l) \Longleftrightarrow m\leq k, n\leq l.$$
For $m,n\in \mathbb{Z}$, we write $|(m,n)|=m+n.$ We also view $\mathbb{Z} \times \mathbb{Z}$ as the additive group in the usual way.
Write $\mathrm{sdim} V$ for the superdimension of a superspace $V$ and $\dim V$ for the dimension of  $V$ as an ordinary linear space. Note that
$\dim V=|\mathrm{sdim} V|.$
Let  $\Pi$ be  the parity functor of  superspaces. Then
 $$\mathrm{sdim} V+\mathrm{sdim}\Pi(V)=(\dim V, \dim V).$$
Moreover, if $W$ is a subsuperspace of a superspace $V$, then
$$ \mathrm{sdim} V/W=\mathrm{sdim}V-\mathrm{sdim}W.$$

In this paper, we write $\mathrm{Ab}(m,n)$ for the abelian Lie superalgebra of superdimension $(m,n).$
As in the Lie algebra case \cite[p. 4302]{Batten-Stitzinger}, we introduce the following definition.
\begin{definition}\label{qqqq}
Let $L$ be a finite-dimensional Lie superalgebra. A Lie superalgebra pair $(K,M)$  is called a defining pair of $L$ provided that $L\cong K/M$ and $M\subset \mathrm{Z}(K)\cap K^{2}$, where $\mathrm{Z}(K)$ is the center of $K$ and $K^{2}:=[K, K]$ is the derived subalgebra of $K$. A defining pair $(K, M)$ of $L$ is said to be  maximal if among all the defining pairs of $L$, $M$  is of a maximal superdimension. In the case $(K,M)$ being a maximal defining pair of $L$, we also call $K$ a cover and $M$ a (Schur) multiplier of $L$.
\end{definition}

The definition makes sense, since one may check as in Lie algebra case (see \cite{Batten-Stitzinger}) that for a finite-dimensional Lie superalgebra,  covers  and multipliers  always exist and they are unique up to Lie superalgebra isomorphism, respectively. Write $\mathcal{C}(L)$ and $\mathcal{M}(L)$  for the cover and multiplier  of Lie superalgebra $L$, respectively.

 As in Lie algebra case \cite{K.Moneyhun}, we will give an upper bound for the superdimension of the multiplier of a Lie superalgebra. To that aim, we first establish the following lemmas.

\begin{lemma}\label{lemma1}
Let $L$ be a Lie superalgebra and suppose $\mathrm{sdim}L/\mathrm{Z}(L)=(m,n).$ Then
$$\mathrm{sdim} L^{2}\leq\left(\frac{1}{2}m(m-1)+\frac{1}{2}n(n+1),mn\right).$$
\end{lemma}
\begin{proof} It is straightforward.
\end{proof}

\begin{lemma}\label{lemma2}
Let $L$ be a Lie superalgebra of $\mathrm{sdim}L=(s,t)$. Then
$$\mathrm{sdim}\mathcal{M}(L)\leq \left(\frac{1}{2}s(s-1)+\frac{1}{2}t(t+1),st\right).$$
\end{lemma}
\begin{proof}
Let $(K,M)$ be a defining pair of $L$ and
suppose $ \mathrm{sdim}K/\mathrm{Z}(K)=(k,l)$. Then $(k,l)\leq \mathrm{sdim}K/M=(s,t)$ and our conclusion follows from  Lemma \ref{lemma1}.
\end{proof}

We should note that a non-super version  of Lemmas \ref{lemma1} and \ref{lemma2} has been given in  \cite[Theorems 3.1 and 3.4]{S.Nayak}.
For a Lie superalgebra $L$ of superdimension $(s, t),$ define the \textit{super-multiplier-rank} of $L$ to be the number pair
$$\mathrm{smr}(L)=\left(\frac{1}{2}s(s-1)+\frac{1}{2}t(t+1),st\right)-\mathrm{sdim}\mathcal{M}(L)$$
and the \textit{multiplier-rank} of $L$ to be
$\mathrm{mr}(L)=|\mathrm{smr}(L)|.$
By Lemma \ref{lemma2}, we have $ \mathrm{smr}(L)\geq(0,0).$

As in the Lie algebra case \cite[Lemma 4 and Theorem 1]{B-M-S}, using the notion of free presentations for Lie superalgebras, one may prove the following two lemmas.

\begin{lemma}\label{lemma4}
Let $L$ be a finite-dimensional Lie superalgebra. Then $L^{2}\cap \mathrm{Z}(L)$ is a homomorphic image of $\mathcal{M}(L/\mathrm{Z}(L)).$
\end{lemma}

\begin{lemma}\label{aaa}
Let $A$  and  $B$  be  finite-dimensional  Lie  superalgebras. Then
$$\mathrm{sdim}\mathcal{M}(A\oplus B)=\mathrm{sdim}\mathcal{M}(A)+\mathrm{sdim}\mathcal{M}(B)+\mathrm{sdim}(A/A^{2}\otimes B/B^{2}).$$
\end{lemma}

\section{Multiplier-rank $0$ nilpotent Lie superalgebras}

The  multiplier-rank $0$ case was also considered  in  \cite[Theorem 3.5]{S.Nayak}, where the multiplier was described  in terms of non-super dimensions. For completeness, we give a proof, which is also somewhat different from the one in \cite[Theorem 3.5]{S.Nayak}.

\begin{proposition}\label{proposition3} Let $L$ be a finite-dimensional Lie superalgebra. Then
$\mathrm{smr}(L)=(0, 0)$  if and only if $L$ is abelian.
\end{proposition}
\begin{proof} Let $L$ be  of superdimension $(m, n)$.
Suppose $L$ is abelian and let $H$ be a superspace with a homogeneous basis
$\{u_{i},x_{k,l},z_{s,t}\mid v_{j},y_{p,q}\},$
where
$$
 1 \leq i\leq m,\ 1\leq k< l\leq m,\ 1\leq s\leq t\leq n,\ 1\leq j\leq n,\
 1 \leq p\leq m,\ 1\leq q\leq n
$$
and
$
|u_{i}|=|x_{k,l}|=|z_{s,t}|=\bar{0},\ |v_{j}|=|y_{p,q}|=\bar{1}.
$
Then $H$ becomes a Lie superalgebra by letting
$$[ u_{k},u_{l}]=x_{k,l},\
 [ u_{p},v_{q}]=y_{p,q},\
 [ v_{s},v_{t}]=z_{s,t}$$
 and the other brackets of basis elements vanish.
Clearly, $L\cong H/H^{2}$. Since $\mathrm{Z}(H)=H^{2}$, one sees that  $H^{2}\subseteq \mathrm{Z}(H)\cap H^{2}$. Hence, $(H,H^{2})$ is a defining pair of $L$ and $\mathrm{smr}(L)=(0, 0)$.

 Conversely, suppose $\mathrm{smr}(L)=(0, 0)$. Then

$$\mathrm{sdim}\mathcal{M}(L)=\left(\frac{1}{2}m(m-1)+\frac{1}{2}n(n+1),mn\right).$$
Let $(K,M)$ be a maximal defining pair of $L$ and suppose $\mathrm{sdim}K/\mathrm{Z}(K)=(k,l).$ Since
 $M\subseteq Z\mathrm{}(K),$  it follows from  Lemma \ref{lemma1} that
\begin{eqnarray*}\mathrm{sdim}K^{2}
\leq\left(\frac{1}{2}m(m-1)+\frac{1}{2}n(n+1), mn\right).
\end{eqnarray*}
Since $M\subset K^2$, we have $M=K^{2}$ and $L\cong K/M$ is abelian.
\end{proof}

\section{Multiplier-rank $1$ nilpotent Lie superalgebras}

   In this section, suppose $L$ is a finite-dimensional non-abelian nilpotent Lie superalgebra and  $\mathrm{sdim}L/\mathrm{Z}(L)=(m,n).$
 Write $\mathrm{Z}_{2}(L)$ for the ideal of $L$ such that $\mathrm{Z}_{2}(L)/\mathrm{Z}(L)=\mathrm{Z}(L/\mathrm{Z}(L)).$
  Suppose $z\in \mathrm{Z}_{2}(L)\backslash \mathrm{Z}(L)$ is a homogeneous element. Then $[L,z]\subset \mathrm{Z}(L)$ is an ideal of $L$. For convenience, write $\lambda(z)=\mathrm{sdim}[L, z]$ and
\begin{equation}\label{qq8}
  \mu(z) =\mathrm{sdim}((L/[L, z])/\mathrm{Z}(L/[L, z])).
 \end{equation}
\begin{lemma}\label{gaodengalgebra}
Suppose $z\in \mathrm{Z}_{2}(L)\backslash \mathrm{Z}(L)$.
\begin{itemize}
\item[(1)] If $|z|=\bar{0},$ then   $\lambda(z)\leq(m-1, n)$ and $\mu(z)\leq(m-1, n).$
\item[(2)] If $|z|=\bar{1},$ then   $\mu(z)\leq(m, n-1).$
\end{itemize}
\end{lemma}
\begin{proof}
(1) In this case, we have the superspace isomorphism $L/\mathrm{Z}_{L}(z)\cong [L,z]$. Since $z\notin \mathrm{Z}(L)$, we have $\mathrm{Z}(L)\subsetneq \mathrm{Z}_{L}(z)$ and $\mathrm{sdim}\mathrm{Z}(L)+(1,0)\leq \mathrm{sdim}\mathrm{Z}_{L}(z).$  Hence
$\lambda(z)=\leq (m-1,n).$ 

Note that $z+[L,z]\in \mathrm{Z}(L/[L,z])$ and $z+[L,z]\notin \mathrm{Z}(L)/[L,z].$ Then
$\mathrm{Z}(L)/[L,z]\subsetneq \mathrm{Z}(L/[L,z])$ and
$$\mathrm{sdim}\mathrm{Z}(L)/[L,z]+(1,0)\leq \mathrm{sdim}\mathrm{Z}(L/[L,z]).$$
By (\ref{qq8}), we have
 $\mu(z)\leq(m-1,n).$

 (2) The proof is similar to the one of (1).

\end{proof}

Recall that $\mathrm{sdim}L/\mathrm{Z}(L)=(m,n).$
Define the \textit{super-derived-rank} of Lie superalgebra $L$ to be
 $$\mathrm{sdr}(L)=\left(\frac{1}{2}m(m-1)+\frac{1}{2}n(n+1),mn\right)-\mathrm{sdim} L^{2}$$
 and the \textit{derived-rank} to be
 $$
 \mathrm{dr}(L)=|\mathrm{sdr}(L)|.
 $$
  It follows from Lemma \ref{lemma1} that $\mathrm{sdr}(L)\geq(0,0)$.
For our purpose, we will first determine all the nilpotent Lie superalgebras $L$ with $\mathrm{dr}(L)\leq1$.
Let $z\in \mathrm{Z}_{2}(L)\backslash \mathrm{Z}(L)$. Suppose $|z|=\bar{1}$. Then
$$
f: L\longrightarrow [L,z],\quad
x\longmapsto [x,z]
$$
is an odd linear epimorphism. Note that $\mathrm{Ker}f=\mathrm{Z}_{L}(z)$. Then we have the following superspace isomorphism:
 $\Pi\left(L/\mathrm{Z}_{L}(z)\right)\cong [L,z].$

\begin{lemma}\label{lemma5}
The following statements holds.
\begin{itemize}
\item[(1)] If the center of $L/\mathrm{Z}(L)$ has a nonzero even part,  then $\mathrm{sdim}(L/\mathrm{Z}(L))^{2}\leq\mathrm{sdr}(L)+(1,0).$
\item[(2)] If the center of $L/\mathrm{Z}(L)$ has a nonzero odd part, then $\mathrm{sdim}(L/\mathrm{Z}(L))^{2}\leq(\mathrm{dr}(L),\mathrm{dr}(L))-\mathrm{sdr}(L)$.
\end{itemize}
\end{lemma}

\begin{proof}
Suppose $z$ is a homogeneous element in $\mathrm{Z}_{2}(L)\backslash \mathrm{Z}(L)$ and $\mu(z)=(b_{1}, b_{2}).$  By Lemma \ref{lemma1}, we have
\begin{equation}\label{eqzmlds33}\mathrm{sdim}L^{2}\leq\left(\frac{1}{2}b_{1}(b_{1}-1)+\frac{1}{2}b_{2}(b_{2}+1), b_{1}b_{2}\right)+\mathrm{sdim}[L,z].\end{equation}
(1) Suppose $|z|=\bar{0}.$ By Lemma \ref{gaodengalgebra}, we have $\mu(z)\leq(m-1, n)$ and it follows from (\ref{eqzmlds33}) that
\begin{equation}\label{qq1}
\lambda(z)\geq(m, n)-\left(\mathrm{sdr}(L)+(1,0)\right).
\end{equation}
Since $L^{2}\subset \mathrm{Z}_{L}(z)$ and $\mathrm{Z}(L)\subset \mathrm{Z}_{L}(z),$ we have $L^{2}+\mathrm{Z}(L)\subset \mathrm{Z}_{L}(z)$. Thus
$$\mathrm{sdim}L/(L^{2}+\mathrm{Z}(L))\geq \mathrm{sdim}L/\mathrm{Z}_{L}(z)$$
 and it follows from (\ref{qq1}) that
\begin{eqnarray*}
\mathrm{sdim}(L/\mathrm{Z}(L))^{2}\leq\mathrm{sdr}(L)+(1,0).
 \end{eqnarray*}

(2) Suppose $|z|=\bar{1}.$ By Lemma \ref{gaodengalgebra}, we have $\mu(z)\leq(m, n-1)$ and  then by (\ref{eqzmlds33}),
\begin{equation*}\label{qq2}
\lambda(z)\geq(n, m)-\mathrm{sdr}(L).
\end{equation*}
 As in (1), we have
 \begin{eqnarray*}
\mathrm{sdim}(L/\mathrm{Z}(L))^{2}
&\leq& 
\mathrm{sdim}L/\mathrm{Z}(L)-\mathrm{sdim}\Pi[L,z]\\
&\leq&(\mathrm{dr}(L),\mathrm{dr}(L))-\mathrm{sdr}(L).
 \end{eqnarray*}
 The proof is complete.
\end{proof}

\begin{lemma}\label{lemma7}
Suppose $z\in \mathrm{Z}_{2}(L)\backslash \mathrm{Z}(L)$ with
  $|z|=\bar{0}$ and $\mathrm{sdim}(L/\mathrm{Z}(L))^{2}=\mathrm{sdr}(L)+(1,0)$. Then $\mathrm{Z}_{L}(z)=L^{2}+\mathrm{Z}(L)$ and $\lambda(z)=(m, n)-\left(\mathrm{sdr}(L)+(1,0)\right)$.
\end{lemma}
\begin{proof}
By (\ref{qq1}), we have
$$\mathrm{sdim}(L/\mathrm{Z}_{L}(z))=\lambda(z)\geq(m,n)-\left(\mathrm{sdr}(L)+(1,0)\right).$$
Hence
\begin{eqnarray*}
\mathrm{sdr}(L)+(1,0)
&=&\mathrm{sdim}(L/\mathrm{Z}(L))-\mathrm{sdim}L/(L^{2}+\mathrm{Z}(L))\\
&\leq& \mathrm{sdim}(L/\mathrm{Z}(L))-\mathrm{sdim}(L/\mathrm{Z}_{L}(z))\\
&\leq&\mathrm{sdr}(L)+(1,0).
\end{eqnarray*}
Therefore,  $\mathrm{sdim}\mathrm{Z}_{L}(z)=\mathrm{sdim}\left(L^{2}+\mathrm{Z}(L)\right)$. Since $L^{2}+\mathrm{Z}(L)\subset\mathrm{Z}_{L}(z)$, we have $L^{2}+\mathrm{Z}(L)=\mathrm{Z}_{L}(z)$. Then
$
\lambda(z)=(m, n)-\left(\mathrm{sdr}(L)+(1,0)\right).
$
The proof is complete.
\end{proof}

Recall that a finite-dimensional Lie superalgebra $\mathfrak{g}$ is called a Heisenberg Lie superalgebra provided that $\mathfrak{g}^{2}=\mathrm{Z}(\mathfrak{g})$ and $\mathrm{sdim}\mathrm{Z}(\mathfrak{g})=(1,0)$ or (0,1). Heisenberg Lie superalgebras consist of two types according to the parity of the central elements (see \cite{RSV}). Suppose   $\mathfrak{g}$ is a Heisenberg Lie superalgebra with $\mathrm{Z}(\frak{g})=\mathbb{F}z$.

\begin{enumerate}
\item[(1)] If $|z|=\bar{0}$, then $\mathfrak{g}$ has a homogeneous basis (called a standard basis)
$$\{u_{1}, u_{2}, \ldots, u_{p}, v_{1}, v_{2}, \ldots, v_{p}, z \mid w_{1}, w_{2}, \ldots, w_{q}\},$$

where
$$
|u_{i}|=|v_{j}|=|z|=\bar{0},\ |w_{k}|=\bar{1}; \ i=1, \ldots, p,\ j=1,\cdots,p,\ k=1, \ldots, q,
$$
and the multiplication is given by
$$
[u_{i}, v_{i}]=-[v_{i}, u_{i}]=z,\ [w_{k}, w_{k}]=z
$$
and the other brackets of basis elements vanishing. Denote by $H(p, q)$ the Heisenberg Lie superalgebra $\frak{g}$ of even center, where $p+q\geq1.$

\item[(2)] If $|z|=\bar{1}$, then $\mathfrak{g}$ has a homogeneous basis (called a standard basis)
$$\{u_{1}, u_{2}, \ldots, u_{k} \mid z, w_{1}, w_{2}, \ldots, w_{k}\},$$

where $$|u_{i}|=\bar{0},\ |w_{j}|=|z|=\bar{1};\ i=1, \ldots, m,\ j=1, \ldots, k,$$
and the multiplication is given by
$$
[u_{i}, w_{i}]=-[w_{i}, u_{i}]=z
$$
and the other brackets of basis elements vanishing. We write $H(k)$ for the Heisenberg Lie superalgebra $\frak{g}$ of odd center, where $k\geq 1$.
\end{enumerate}

\begin{proposition}\label{hsda}
Let $H(p,q)$ be a Heisenberg Lie superalgebra of even  center. Then
\begin{equation*}
\mathrm{sdim}\mathcal{M}(H(p,q))=\left\{
               \begin{array}{ll}
                 (2p^{2}-p+\frac{1}{2}q^{2}+\frac{1}{2}q-1, 2pq), & \hbox{$p+q\geq2$}\\
                 (0,0), & \hbox{$p=0,q=1$}\\
                 (2,0), & \hbox{$p=1,q=0$}.\\

        \end{array}
       \right.
\end{equation*}
\end{proposition}
\begin{proof} We only consider the case $p=0, q=1$, while the remaining cases may be argued as in \cite[Theorem 4.3]{S.Nayak}.     Suppose $(K,M)$ is a defining pair of $H(0,1)$. Then $K/M\cong H(0,1)$ and $K/M$ has a standard basis
$\{a+M\mid b+M\}$,
where $a, b\in K$ with $|a|=\bar{0},$ $|b|=\bar{1}$.  Then $[b, b]\equiv a\pmod{M}$. Since $[[b, b], b]=0,$ one sees that $[a,b]=0.$ It follows that
$K^2$ is 1-dimensional and  not contained in $M$. Since $M\subset K^2$, we have $M=0$.
 The proof is complete.
\end{proof}

One can determine the multiplier for Heisenberg Lie superalgebras of odd  center (see \cite{miao-liu}).

\begin{proposition}\label{hsda111}
 Let $H(k)$ be a Heisenberg Lie superalgebra of odd  center. Then
\begin{equation*}
\mathrm{sdim}\mathcal{M}(H(k))=\left\{
               \begin{array}{ll}
                 (k^{2}, k^{2}-1), & \hbox{$k\geq2$}\\
                 (1,1), & \hbox{$k=1$}.\\
                 \end{array}
       \right.
\end{equation*}
\end{proposition}

\begin{lemma}\label{lemma12}
If $\mathrm{sdr}(L)=(0,0)$, then $L/\mathrm{Z}(L)$ is either abelian or isomorphic to $H(1,0)$.
\end{lemma}
\begin{proof} If $\mathrm{Z}_{2}(L)/\mathrm{Z}(L)$ has a nonzero odd part, then by Lemma \ref{lemma5}(2), we have $\mathrm{sdim}(L/\mathrm{Z}(L))^{2}=(0,0)$ and hence $L/\mathrm{Z}(L)$ is abelian.

 Suppose the odd part of $\mathrm{Z}_{2}(L)/\mathrm{Z}(L)$ is zero.
 By Lemma \ref{lemma5}(1),  $\mathrm{sdim}(L/\mathrm{Z}(L))^{2}=(0,0)$ or $(1,0)$.
If $\mathrm{sdim}(L/\mathrm{Z}(L))^{2}=(0, 0)$, then $L/\mathrm{Z}(L)$ is abelian. 
So we suppose $\mathrm{sdim}(L/\mathrm{Z}(L))^{2}=(1, 0).$
Then
$(L/\mathrm{Z}(L))^{2}\subset \mathrm{Z}(L/\mathrm{Z}(L))$. We claim  that $\mathrm{sdim}\mathrm{Z}(L/\mathrm{Z}(L))=(1, 0)$. For any even element $x\in \mathrm{Z}_{2}(L)\backslash\mathrm{Z}(L)$, by Lemma \ref{gaodengalgebra}(1), we have $\mu(x)\leq (m-1,n)$. Since $\mathrm{sdr}(L)=(0,0)$, $(L/[L,x])^{2}=L^{2}/[L,x]$ and $\mu(x)\leq (m-1,n)$, it follows from Lemma \ref{lemma1} that $\lambda(x)\geq(m-1,n)$.  By Lemma \ref{gaodengalgebra}(1), we also have $\lambda(x)\leq(m-1, n)$.  Therefore, $\lambda(x)=(m-1, n)$. Then, since $[L,x]\cong L/\mathrm{Z}_{L}(x)$,  we have $$\mathrm{sdim}\ \mathrm{Z}_{L}(x)/\mathrm{Z}(L)=\mathrm{sdim}(L/\mathrm{Z}(L))-\lambda(x)=(1, 0).$$
Let $y\in\mathrm{Z}_{2}(L)\backslash\mathrm{Z}(L)$ be even. We have
\begin{eqnarray}
\mathrm{Z}(L)\subset \mathrm{Z}_{L}(x)\cap \mathrm{Z}_{L}(y)\subset \mathrm{Z}_{L}(x).
\end{eqnarray}
If $\mathrm{Z}(L)=\mathrm{Z}_{L}(x)\cap \mathrm{Z}_{L}(y),$ since $L^{2}\subset\mathrm{Z}_{L}(x)\cap \mathrm{Z}_{L}(y)$, we have  $L^{2}\subset \mathrm{Z}(L)$ and $\mathrm{sdim}(L/\mathrm{Z}(L))^{2}=(0, 0)$, a contrdiction. Hence $\mathrm{Z}_{L}(x)\cap \mathrm{Z}_{L}(y)=\mathrm{Z}_{L}(y)$ and then $ \mathrm{Z}_{L}(x)=\mathrm{Z}_{L}(y).$ Since $\mathrm{sdim}\mathrm{Z}_{L}(x)/\mathrm{Z}(L)=(1,0)$, one sees that $x+\mathrm{Z}(L)$ and $y+\mathrm{Z}(L)$ are  linearly dependent.
Hence $\mathrm{sdim} \mathrm{Z}(L/\mathrm{Z}(L))=(1,0)$ and $L/\mathrm{Z}(L)$  is a Heisenberg Lie superalgebra of even  center.
Suppose $L/\mathrm{Z}(L)\cong H(p,q)$, where $p+q\geq 1.$
Assume that $p+q\geq 2$.  By Proposition \ref{hsda},
$$\mathrm{sdim}\mathcal{M}(L/\mathrm{Z}(L))=(2p^{2}-p+\frac{1}{2}q^{2}+\frac{1}{2}q-1, 2pq)$$
and by Lemma \ref{lemma4},
\begin{eqnarray}\label{08}
 \mathrm{sdim}L^{2} \leq \mathrm{sdim}(L/\mathrm{Z}(L))^{2}+\mathrm{sdim}\mathcal{M}(L/\mathrm{Z}(L)).
\end{eqnarray}
Then $\mathrm{sdim}L^{2}<(2p^{2}+p+\frac{1}{2}q^{2}+\frac{1}{2}q, 2pq+q).$
 However, since $\mathrm{sdr}(L)=(0,0)$, we have
$$\mathrm{sdim}L^{2}=(2p^{2}+p+\frac{1}{2}q^{2}+\frac{1}{2}q, 2pq+q),$$
a contradiction. Assume that $p=0$ and $q=1.$ By Proposition \ref{hsda},   $\mathrm{sdim}\mathcal{M}(H(0,1))=(0,0)$.
Then by 
(\ref{08}) and $\mathrm{sdr}(L)=(0,0)$, we have (1, 1)=$\mathrm{sdim}L^{2}\leq (1, 0),$
a contradiction. Summarizing,  $L/\mathrm{Z}(L)$ is abelian or isomorphic to $H(1, 0).$
\end{proof}

\begin{lemma}\label{lemma8}
Suppose $L/\mathrm{Z}(L)$ is a Heisenberg Lie superalgebra.
\begin{itemize}
\item[(1)]If $\mathrm{sdr}(L)=(1,0)$, then $L/\mathrm{Z}(L)\cong H(1,0).$
\item[(2)]If $\mathrm{sdr}(L)=(0,1)$, then $L/\mathrm{Z}(L)\cong H(0,1).$
\end{itemize}
\end{lemma}
\begin{proof}
(1) Suppose $L/\mathrm{Z}(L)= H(p,q)$, where $p+q\geq 1.$
Assume that $p+q\geq2.$
Then by 
(\ref{08}) and Proposition \ref{hsda},
for $p\neq0$, we have
\begin{eqnarray*}
\mathrm{sdim}L^{2}
&<&\mathrm{sdim}L/\mathrm{Z}(L)-(1, 0)+\mathrm{sdim}\mathcal{M}(L/\mathrm{Z}(L))\\
&=&(2p^{2}+p+\frac{1}{2}q^{2}+\frac{1}{2}q-1, 2pq+q),
\end{eqnarray*}
and for $p=0$, similarly, we have
$
\mathrm{sdim}L^{2}
<(2p^{2}+p+\frac{1}{2}q^{2}+\frac{1}{2}q, 2pq+q-1).
$
However, since $\mathrm{sdr}(L)=(1,0)$, we have $\mathrm{sdim}L^{2}=\left(2p^{2}+p+\frac{1}{2}q^{2}+\frac{1}{2}q-1, 2pq+q\right)$, a contradiction.
Assume that $p=0$ and $q=1$. By Proposition \ref{hsda}, $\mathrm{sdim}\mathcal{M}(H(0,1))=(0,0)$. Then by $\mathrm{sdr}(L)=(1,0)$ and 
(\ref{08}), we have (0, 1) = $\mathrm{sdim}L^{2}\leq (1, 0)$,
a contradiction.

Suppose $L/\mathrm{Z}(L)= H(k)$, where $k\geq 1.$
Assume that $k>1$. By Proposition \ref{hsda111}, we have $\mathrm{sdim}\mathcal{M}(L/\mathrm{Z}(L))=(k^{2}, k^{2}-1)$.
Then by (\ref{08}), 
 we have
\begin{eqnarray*}
\mathrm{sdim}L^{2}
&<&\mathrm{sdim}L/\mathrm{Z}(L)-(0,1)+\mathrm{sdim}\mathcal{M}(L/\mathrm{Z}(L))\\
&=&(k^{2}+k,k^{2}+k-1).
 \end{eqnarray*}
However, since $\mathrm{sdr}(L)=(1,0)$, we have $\mathrm{sdim}L^{2}=(k^{2}+k, k^{2}+k)$, a contradiction.
Assume that  $k=1$. By Proposition \ref{hsda111}, we have $\mathrm{sdim}\mathcal{M}(H(1))=(1,1).$ Then by $\mathrm{sdr}(L)=(1,0)$ and (\ref{08}), 
 we have
\begin{eqnarray*}
(2,2)=\mathrm{sdim}L^{2}
<\mathrm{sdim}L/\mathrm{Z}(L)-(0, 1)+\mathrm{sdim}\mathcal{M}(L/\mathrm{Z}(L))
=(2,2),
\end{eqnarray*}
a contradiction.
Summarizing, we have $L/\mathrm{Z}(L)=H(1,0).$

(2) Suppose $L/\mathrm{Z}(L)= H(p,q)$, where $p+q\geq 1.$
Assume that $p+q\geq2.$  Then by (\ref{08})  
 and Proposition \ref{hsda}, we have
\begin{eqnarray}\label{hua}
\mathrm{sdim}L^{2}
\leq\left(2p^{2}-p+\frac{1}{2}q^{2}+\frac{1}{2}q, 2pq\right).
\end{eqnarray}
Then by (\ref{hua}) and $\mathrm{sdr}(L)=(0,1)$, we have
$p+q\leq1$, a contradiction.
Assume that $p=1,q=0$. Since $\mathrm{sdr}(L)=(0,1)$, we have $\mathrm{sdim}L^{2}=(3, -1)$, contradicting the assumption that $\mathrm{sdim}L^{2}>(0,0).$

Suppose $L/\mathrm{Z}(L)= H(k)$, where $k\geq 1.$
Assume that $k>1$. Then by (\ref{08}) 
and Proposition \ref{hsda111}, we have
\begin{eqnarray*}\mathrm{sdim}L^{2}
<\mathrm{sdim}L/\mathrm{Z}(L)-(0, 1)+\mathrm{sdim}\mathcal{M}(L/\mathrm{Z}(L))=(k^{2}+k, k^{2}+k-1).
\end{eqnarray*}
However, since $\mathrm{sdr}(L)=(0,1)$, we have
$\mathrm{sdim}L^{2}=(k^{2}+k+1, k^{2}+k-1),$ a contradiction.
Assume that $k=1$. By Proposition \ref{hsda111}, we have $\mathrm{sdim}\mathcal{M}(H(1))=(1,1).$ Since $\mathrm{sdr}(L)=(0,1)$, then by (\ref{08}), 
we have
\begin{eqnarray*}
(3,1)=\mathrm{sdim}L^{2}
<\mathrm{sdim}L/\mathrm{Z}(L)+\mathrm{sdim}\mathcal{M}(L/\mathrm{Z}(L))=(2,3),
\end{eqnarray*} a contradiction.
Summarizing, we have $L/\mathrm{Z}(L)=H(0,1).$
\end{proof}

The following proposition is analogues to \cite[Theorem 3]{B-M-S}.

\begin{proposition}\label{13}
Suppose $L$ is a non-abelian nilpotent Lie superalgebra. Then
\begin{itemize}
\item[(1)] $\mathrm{smr}L\neq (0,1).$
\item[(2)]
$\mathrm{smr}(L)=(1,0)$ if and only if $L\cong H(1,0)$.
\end{itemize}
\end{proposition}
\begin{proof} Let $\mathrm{sdim}L=(s,t)$.

(1) Assume conversely that $\mathrm{smr}(L)=(0,1).$ By Proposition \ref{proposition3}, $L$ is not abelian. Let $(K,M)$ be a maximal defining pair of $L$.  Since $\mathrm{smr}(L)=(0,1)$, we have
\begin{equation}\label{00}
\mathrm{sdim}M=\left(\frac{1}{2}s(s-1)+\frac{1}{2}t(t+1), st\right)-(0, 1).
\end{equation}
We claim that  $M=\mathrm{Z}(K)$. If not,
since $M\subsetneq \mathrm{Z}(K),$  we have $\mathrm{sdim}(K/\mathrm{Z}(K))<\mathrm{sdim}K/M=(s, t)$. Hence $\mathrm{sdim}(K/\mathrm{Z}(K))\leq (s-1, t)$ or $\mathrm{sdim}(K/\mathrm{Z}(K))\leq(s, t-1)$. Suppose $\mathrm{sdim}(K/\mathrm{Z}(K))\leq (s-1, t)$. Then by Lemma \ref{lemma1}, we have
$$\mathrm{sdim}K^{2}\leq \left(\frac{1}{2}(s-1)(s-2)+\frac{1}{2}t(t+1), (s-1)t\right).$$
Since $M\subset K^{2}$, we have
$\mathrm{sdim}L\leq(1,1).$ Since $L$ is not abeian, we must have $\mathrm{sdim}L=(1,1).$ It is easy to deduce that
$L\cong H(0,1)$. Consequently, $\mathrm{smr}(L)=(1,1),$ contradicting the assumption that $\mathrm{smr}(L)=(0,1).$

Suppose $\mathrm{sdim}K/\mathrm{Z}(K)\leq(s, t-1)$. Then by Lemma \ref{lemma1} and (\ref{00}), we have
$\mathrm{sdim}L\leq(1,0)$, contradicting the assumption that $L$ is not abeian.
 Hence $M=\mathrm{Z}(K)$ and $\mathrm{sdim}K/\mathrm{Z}(K)=(s,t).$ 
Since $L$ is not abelian, we have $M\subsetneq K^{2}$. 
 So we have $\mathrm{sdr}(K)=(0,0)$.  By Lemma \ref{lemma12}, $L\cong K/\mathrm{Z}(K)$ is abelian or $H(1,0)$. Then   $\mathrm{smr}(L)= (0,0)$ or $(1,0)$, a contradiction.

(2) Suppose $L\cong H(1,0)$.  By Proposition \ref{hsda}, we have $\mathrm{sdim}\mathcal{M}(L)=(2,0)$  and hence $\mathrm{smr}(L)=(1,0).$

Suppose  $\mathrm{smr}(L)=(1, 0).$ By Proposition \ref{proposition3}, $L$ is not abelian. Let $(K,M)$ be a maximal defining pair of $L.$
As in (1), we have $M=\mathrm{Z}(K).$ 
Since $L$ is not abelian,  we have $M \subsetneq K^{2}$.
It follows that $\mathrm{sdr}(K)=(0,0)$. By Lemma \ref{lemma12}, $L\cong K/\mathrm{Z}(K)=H(1,0)$.
\end{proof}

\section{Multiplier-rank $2$ nilpotent Lie superalgebras}
In this section, suppose $L$ is a finite-dimensional non-abelian nilpotent Lie superalgebra  and  $\mathrm{sdim}L/\mathrm{Z}(L)=(m,n).$
Let us   establish several technical lemmas.

\begin{lemma}\label{9} Suppose $z\in \mathrm{Z}_{2}(L)\backslash \mathrm{Z}(L)$ is an even element and
$\lambda(z)=(m, n)-\left(\mathrm{sdr}(L)+(1,0)\right).$ Then $\mu(z)=(m-1, n)$. Moreover, $L/[L,z]/\mathrm{Z}(L/[L,z])$ is either $\mathrm{Ab}(m-1,n)$ or $H(1,0)$.
\end{lemma}
\begin{proof}
Suppose $\mu(z)=(b_{1},b_{2})$.
 By Lemma \ref{gaodengalgebra}, we have $\mu(z)\leq (m-1, n).$  By Lemma \ref{lemma1},
\begin{eqnarray*}
&&\left(\frac{1}{2}m(m-1)+\frac{1}{2}n(n+1), mn\right)-\mathrm{sdr}(L)\\
&&=\mathrm{sdim}L^{2}\\
&&\leq \left(\frac{1}{2}b_{1}(b_{1}-1)+\frac{1}{2}b_{2}(b_{2}+1), b_{1}b_{2}\right)+\lambda(z)\\
&&=\left(\frac{1}{2}m(m-1)+\frac{1}{2}n(n+1), mn\right)-\mathrm{sdr}(L).
\end{eqnarray*}
Therefore,  $\mu(z)=(m-1, n)$ and $\mathrm{sdr}(L/[L,z])=(0,0).$
Then our conclusion follows from Lemma \ref{lemma12}.
\end{proof}

\begin{lemma}\label{balabala}
Let $\mathrm{sdr}(L)=(1, 0).$ Then $L/\mathrm{Z}(L)$ is isomorphic to one of the following Lie superalgebras:
\begin{itemize}
\item[(1)] an abelian Lie superalgebra;
\item[(2)] $H(1,0)$;
\item[(3)] $H(1,0)\oplus \mathrm{Ab}(1,0)$;
\item[(4)] the Lie algebra with basis $\{x,y,z,t\}$ and multiplication given by
$$[x, y]=-[y, x]=z,\ [x, z]=-[z, x]=t$$
and the other brackets of basis elements vanishing.
\end{itemize}
\end{lemma}

\begin{proof} Our argument is divided into two parts.

(I) Suppose $\mathrm{Z}_{2}(L)/\mathrm{Z}(L)$ has a nonzero odd part. Then by Lemma \ref{lemma5}(2), we have $\mathrm{sdim}(L/\mathrm{Z}(L))^{2}=(0,0)$ or $(0,1).$ If $\mathrm{sdim}(L/\mathrm{Z}(L))^{2}=(0, 0)$, then $L/\mathrm{Z}(L)$ is abelian. Thus we suppose $\mathrm{sdim}(L/\mathrm{Z}(L))^{2}=(0,1).$
Then 
$(L/\mathrm{Z}(L))^{2}\subset \mathrm{Z}(L/\mathrm{Z}(L))$. If $\mathrm{sdim}\mathrm{Z}(L/\mathrm{Z}(L))=(0,1)$, then  $L/\mathrm{Z}(L)$ is a Heisenberg superalgebra of odd  center. Then by Lemma \ref{lemma8}(1), we have $\mathrm{sdr}(L)\neq(1,0),$ contradicting the assumption. Then we can assume that $\mathrm{sdim}\mathrm{Z}(L/\mathrm{Z}(L))=(k,l+1)>(0,1).$ Suppose $S$ is a subsuperspace of $L/\mathrm{Z}(L)$ such that $(L/\mathrm{Z}(L))^{2}\oplus S=\mathrm{Z}(L/\mathrm{Z}(L)).$ Suppose $T$ is a  subsuperspace such that
 $((L/\mathrm{Z}(L))^{2}\oplus S) \oplus T=L/\mathrm{Z}(L).$ Write $H=(L/\mathrm{Z}(L))^{2}+T$. Then $H\lhd L/\mathrm{Z}(L)$ and   it is easy to deduce that $H^{2}=(L/\mathrm{Z}(L))^{2}=\mathrm{Z}(H)$. Hence $H\cong H(p)$ for some $p\geq1$. Then $\mathrm{sdim}S=(k,l)$ and $(m,n)-(k,l)=(p,p+1)$.

Assume that $p>1.$ Since $\mathrm{sdr}(L)=(1, 0),$   by Propositions \ref{hsda111}, 
(\ref{08}) and Lemma \ref{aaa}, we have
\begin{eqnarray*}
&&\left(\frac{1}{2}m(m-1)+\frac{1}{2}n(n+1),mn\right)-(1,0)\\
&&=\mathrm{sdim}L^{2}\\
&&\leq\mathrm{sdim}\mathcal{M}\left(L/\mathrm{Z}(L)\right)+(0,1)\\
&&=\left(\frac{1}{2}k(k-1)+\frac{1}{2}l(l+1), kl\right)+\left(p^{2},p^{2}-1\right)+\left(pk+pl, pk+pl\right)+(0, 1).
\end{eqnarray*}
Substituting $m=p+k$ and $n=p+1+l,$ one may obtain that  $p+l\leq0,$ contradicting the assumption that $p+l>0$.

Assume that $p=1$. 
As in the case $p>1$, one may obtain that  $l\leq-1$, contradicting the assumption that $l\geq0$.

(II) Suppose the odd part of $\mathrm{Z}_{2}(L)/\mathrm{Z}(L)$ is zero. By Lemma \ref{lemma5}(1), we have $\mathrm{sdim}(L/\mathrm{Z}(L))^{2}=(0,0)$, $(1,0)$ or $(2,0)$. If $\mathrm{sdim}(L/\mathrm{Z}(L))^{2}=(0, 0)$, then $L/\mathrm{Z}(L)$ is abelian. Suppose  that $\mathrm{sdim}(L/\mathrm{Z}(L))^{2}=(1,0)$. If $\mathrm{sdim}\mathrm{Z}(L/\mathrm{Z}(L))=(1,0),$ then by Lemma \ref{lemma8}(1), we have $L/\mathrm{Z}(L)\cong H(1,0)$.
Since $\mathrm{sdim}(L/\mathrm{Z}(L))^{2}=(1,0)$ and hence $(L/\mathrm{Z}(L))^{2}\subset \mathrm{Z}(L/\mathrm{Z}(L)),$ we can assume that $\mathrm{sdim}\mathrm{Z}(L/\mathrm{Z}(L))=(k+1,0)>(1, 0)$.
Let $S$ be a  subsuperspace of $L/\mathrm{Z}(L)$ such that $(L/\mathrm{Z}(L))^{2}\oplus S=\mathrm{Z}(L/\mathrm{Z}(L)).$ Suppose $H$ is a subsuperspace containing $(L/\mathrm{Z}(L))^{2}$ such that $H\oplus S=L/\mathrm{Z}(L).$   Then $H$ is a subsuperalgebra of $L/\mathrm{Z}(L)$ and  $H^{2}=(L/\mathrm{Z}(L))^{2}=\mathrm{Z}(H)$. Since $\mathrm{sdim}(L/\mathrm{Z}(L))^{2}=(1, 0)$, we have $H\cong H(p,q)$, where $p+q\geq1.$ Then $\mathrm{sdim}S=(k,0)$ and $(m-k, n)=(2p+1, q).$

Assume that $p+q\geq2.$
Since $\mathrm{sdr}(L)=(1, 0),$ by Propositions \ref{hsda}, 
(\ref{08}) and Lemma \ref{aaa}, we have
\begin{eqnarray*}
&&\left(\frac{1}{2}m(m-1)+\frac{1}{2}n(n+1),mn\right)-(1,0)\\
&&=\mathrm{sdim}L^{2}\\
&&\leq \mathrm{sdim}\mathcal{M}(L/\mathrm{Z}(L))+(1,0)\\
&&=\left(\frac{1}{2}k(k-1), 0\right)+\left(2p^{2}-p+\frac{1}{2}q^{2}+\frac{1}{2}q-1, 2pq\right)+(2pk, kq)+(1,0).
\end{eqnarray*}
Substituting $m=2p+1+k$ and $n=q,$ one may obtain that  $p+q=0,$ contradicting the assumption that $p+q\geq2$.

Assume that $p=1$ and $q=0$. 
As in the case $p+q\geq2$, one may obtain that  $k=1$. Hence $L/\mathrm{Z}(L)\cong H(1,0)\oplus \mathrm{Ab}(1,0).$

Assume that $p=0$ and $q=1$. 
As in  the case $p+q\geq2$, one may obtain that  $1\leq0$, a contradiction.

Now suppose  $\mathrm{sdim}(L/\mathrm{Z}(L))^{2}=(2,0).$ For any even element $x\in \mathrm{Z}_{2}(L)\backslash\mathrm{Z}(L)$, by Lemma \ref{lemma7}, we have $L^{2}+\mathrm{Z}(L)=\mathrm{Z}_{L}(x)$ and $\lambda(x)=(m-2,n)$.
Since
$x\in \mathrm{Z}_{L}(x)=L^{2}+\mathrm{Z}(L),$
we have $\mathrm{Z}_{2}(L)\subset \mathrm{Z}_{L}(x)=L^{2}+\mathrm{Z}(L)$.
Let us show  that $\mathrm{sdim}\mathrm{Z}(L/\mathrm{Z}(L))=(1,0).$ If not, we have $\mathrm{sdim} \mathrm{Z}_{2}(L)/\mathrm{Z}(L)=(2, 0)$, since
$$\mathrm{sdim} \mathrm{Z}_{2}(L)/\mathrm{Z}(L)\leq \mathrm{sdim} \mathrm{Z}_{L}(x)/\mathrm{Z}(L)
=(2,0).$$
Then $\mathrm{Z}_{L}(x)=\mathrm{Z}_{2}(L)$  for all $x\in \mathrm{Z}_{2}(L)\backslash\mathrm{Z}(L)$. Therefore,  $(L/\mathrm{Z}(L))^{2}=\mathrm{Z}(L/\mathrm{Z}(L)).$
Since $\langle x,\mathrm{Z}(L)\rangle/[L,x]\subset\mathrm{Z}(L/[L,x])$, by Lemma \ref{9}, we have
\begin{eqnarray*}
(m-1,n)&=&\mathrm{sdim}(L/\mathrm{Z}(L))/(\langle x,\mathrm{Z}(L)\rangle/\mathrm{Z}(L))\\
&\geq& \mathrm{sdim}(L/[L,x])/\mathrm{Z}(L/[L,x])\\
&=&(m-1,n).
\end{eqnarray*}
Then we have the following Lie superalgebra isomorphism:
 $$(L/\mathrm{Z}(L))/(\langle x,\mathrm{Z}(L)\rangle/\mathrm{Z}(L))\cong (L/[L,x])/\mathrm{Z}(L/[L,x]).$$
Then by Lemma \ref{9}, $L/\mathrm{Z}(L)/\langle x, \mathrm{Z}(L)\rangle/\mathrm{Z}(L)$ is abelian or isomorphic to $H(1,0).$ However, since $\mathrm{sdim}(L/\mathrm{Z}(L))^{2}=(2,0),$  one sees that $L/\mathrm{Z}(L)/\langle x,\ \mathrm{Z}(L)\rangle/\mathrm{Z}(L)$ is not abelian.
Thus $L/\mathrm{Z}(L)/\langle x, \mathrm{Z}(L)\rangle/\mathrm{Z}(L)\cong H(1,0).$ Then it is routine to deduce that $\mathrm{sdim}\mathrm{Z}(L/\mathrm{Z}(L))=(1,0)$, a contradiction.

Suppose $\mathrm{sdim}\mathrm{Z}(L/\mathrm{Z}(L))=(1,0)$. Let $x$ and $\mathrm{Z}(L)$ generate $\mathrm{Z}_{2}(L)$. By Lemma \ref{lemma7}, we have $L^{2}+\mathrm{Z}(L)=\mathrm{Z}_{L}(x)\supsetneq \mathrm{Z}_{2}(L)$, since $\mathrm{sdim}\mathrm{Z}_{L}(x)/\mathrm{Z}(L)=(2, 0).$
By Lemma \ref{9}, we have $\mu(x)=(m-1, n)$. Clearly, $\mathrm{Z}_{2}(L)/[L,x]\subset \mathrm{Z}(L/[L,x])$. Then
\begin{eqnarray*}
(m-1,n)&=&\mathrm{sdim}L/\mathrm{Z}(L)/\mathrm{Z}_{2}(L)/\mathrm{Z}(L)\\
&\geq&\mathrm{sdim}L/[L,x]/\left(\mathrm{Z}\left(L/[L,x]\right)\right)\\
&=&(m-1,n).
\end{eqnarray*}
Therefore, $\mathrm{Z}(L/[L,x]) = \mathrm{Z}_{2}(L)/[L,x]$. By Lemma \ref{9}, we have
$$L/\mathrm{Z}(L)/\mathrm{Z}(L/\mathrm{Z}(L))\cong L/\mathrm{Z}_{2}(L)\cong L/[L,x]/\mathrm{Z}(L/[L,x])$$
is isomorphic to $H(1,0)$, since $L/\mathrm{Z}_{2}(L)$ is not abelian. Hence  $L/\mathrm{Z}(L)$ is isomorphic to the Lie algebra in (4).
 \end{proof}

\begin{lemma}\label{balabalala}
Let $\mathrm{sdr}(L)=(0, 1)$. Then $L/\mathrm{Z}(L)$ is isomorphic to one of the following Lie superalgebras:
\begin{itemize}
\item[(1)] An abelian Lie superalgebra;
\item[(2)] $H(0, 1)$;
\item[(3)] $H(1, 0)\oplus \mathrm{Ab}(0,1)$.
\end{itemize}
\end{lemma}

\begin{proof}Since $\mathrm{sdr}(L)=(0,1)$,  by Lemma \ref{lemma5}, we have $\mathrm{sdim}(L/\mathrm{Z}(L))^{2}\leq(1,1)$. If $\mathrm{sdim}(L/\mathrm{Z}(L))^{2}=(0, 0)$, then $L/\mathrm{Z}(L)$ is abelian.

Suppose $\mathrm{sdim}(L/\mathrm{Z}(L))^{2}=(1, 0)$. If $\mathrm{sdim}\mathrm{Z}(L/\mathrm{Z}(L))=(1, 0)$, then by Lemma \ref{lemma8}, we have $L/\mathrm{Z}(L)\cong H(0,1)$.  Suppose $\mathrm{sdim}\mathrm{Z}(L/\mathrm{Z}(L))=(k+1, l)>(1,0).$ Since $\mathrm{sdim}(L/\mathrm{Z}(L))^{2}=(1, 0)$, we have $(L/\mathrm{Z}(L))^{2}\subset \mathrm{Z}(L/\mathrm{Z}(L))$. Let $S$  be a subsuperspace of $L/\mathrm{Z}(L)$ such that $(L/\mathrm{Z}(L))^{2}\oplus S =\mathrm{Z}(L/\mathrm{Z}(L)).$  Suppose $H$ is a subsuperspace containing $(L/\mathrm{Z}(L))^{2}$ such that $H\oplus S=L/\mathrm{Z}(L).$   Then $H$ is a subsuperalgebra of $L/\mathrm{Z}(L)$ and  $H^{2}=(L/\mathrm{Z}(L))^{2}=\mathrm{Z}(H)$. Since $\mathrm{sdim}(L/\mathrm{Z}(L))^{2}=(1, 0)$, we have $H\cong H(p,q)$, where $p+q\geq1.$ Then $\mathrm{sdim}S=(k, l)$ and $(m-k, n-l)=(2p+1, q)$.

Assume that $p+q\geq2.$
Since $\mathrm{sdr}(L)=(0, 1)$, by Propositions \ref{hsda}, 
(\ref{08}) and Lemma \ref{aaa}, we have
\begin{eqnarray*}
&&\left(\frac{1}{2}m(m-1)+\frac{1}{2}n(n+1), mn\right)-(0, 1)\\
&&=\mathrm{sdim}L^{2}\\
&&\leq \mathrm{sdim}\mathcal{M}(L/\mathrm{Z}(L))+(1, 0)\\
&&=\left(\frac{1}{2}k(k-1)+\frac{1}{2}l(l+1), kl\right)+\left(2p^{2}-p+\frac{1}{2}q^{2}+\frac{1}{2}q-1, 2pq\right)\\
&&\ \ \ \ \ +(2pk+ql,\ kq+2pl)+(1, 0).
\end{eqnarray*} 
Substituting $m=2p+1+k$ and $n=q+l$, one may obtain that $p+q\leq1$, contradicting the assumption that $p+q\geq2$.

Assume that $p=1$, $q=0$. 
As in  the case  $p+q\geq2$, one gets $k=0$ and $l=1$. Hence $L/\mathrm{Z}(L)\cong H(1, 0)\oplus \mathrm{Ab}(0,1)$.

Assume that $p=0$, $q=1$. 
As in  the case  $p+q\geq2$, one may obtain that  $k+l=0$, contradicting the assumption that $k+l\geq1$.

Suppose $\mathrm{sdim}(L/\mathrm{Z}(L))^{2}=(0,1).$ Suppose the odd part of $\mathrm{Z}_{2}(L)/\mathrm{Z}(L)$ is zero.
Then $(L/\mathrm{Z}(L))^{2}\subset \mathrm{Z}(L/\mathrm{Z}(L))$, contradicting the assumption that the odd part of $\mathrm{Z}_{2}(L)/\mathrm{Z}(L)$ is zero.

Suppose $\mathrm{sdim}(L/\mathrm{Z}(L))^{2}=(1,1)$. Suppose the odd part of $\mathrm{Z}_{2}(L)/\mathrm{Z}(L)$ is zero. For any even element $x\in\mathrm{Z}_{2}(L)\backslash\mathrm{Z}(L)$. By Lemma \ref{lemma7}, $L^{2}+\mathrm{Z}(L)=\mathrm{Z}_{L}(x)$ and $\lambda(x)=(m-1, n-1).$ Hence  $\mathrm{sdim}\mathrm{Z}_{L}(x)/\mathrm{Z}(L)=(1, 1).$
Now $x\in\mathrm{Z}_{L}(x)$ for all $x\in\mathrm{Z}_{2}(L)\backslash\mathrm{Z}(L).$ Hence $\mathrm{Z}_{2}(L)\subsetneq \mathrm{Z}_{L_{}}(x)$ for all $x\in \mathrm{Z}_{2}(L)\backslash\mathrm{Z}(L).$ We claim that $L/\mathrm{Z}_{2}(L)$ is not abelian. If not, $L^{2}\subset\mathrm{Z}_{2}(L)$, since $\mathrm{Z}(L)\subset\mathrm{Z}_{2}(L)$, we have $\mathrm{Z}_{L}(x)
\subset\mathrm{Z}_{2}(L)$, contradicting the assumption that $\mathrm{Z}_{2}(L)\subsetneq \mathrm{Z}_{L_{}}(x)$.
 Assert that $\mathrm{sdim}\mathrm{Z}_{2}(L)/\mathrm{Z}(L)=(1, 0).$ Since the odd part of $\mathrm{Z}_{2}(L)/\mathrm{Z}(L)$ is zero, for any even element $x\in \mathrm{Z}_{2}(L)\backslash\mathrm{Z}(L)$,  we have $$\mathrm{sdim}\mathrm{Z}_{2}(L)/\mathrm{Z}(L)\leq\mathrm{sdim}\mathrm{Z}_{L}(x)/\mathrm{Z}(L)=(1, 1).$$
 Suppose $x$ and $\mathrm{Z}(L)$ generate $\mathrm{Z}_{2}(L)$.
By Lemma \ref{9}, $\mu(x)=(m-1, n)$. Clearly, $\mathrm{Z}_{2}(L)/[L,x]\subset \mathrm{Z}(L/[L,x])$. Note that
\begin{eqnarray*}
(m-1,n)&=&\mathrm{sdim}L/\mathrm{Z}(L)/\mathrm{Z}_{2}(L)/\mathrm{Z}(L)\\
&\geq&\mathrm{sdim}L/[L,x]/(\mathrm{Z}(L/[L,x]))\\
&=&(m-1,n).
\end{eqnarray*}
We have $\mathrm{Z}(L/[L,x])=\mathrm{Z}_{2}(L)/[L,x]$. Then by Lemma \ref{9}, we have
$$L/\mathrm{Z}(L)/\mathrm{Z}(L/\mathrm{Z}(L))\cong L/\mathrm{Z}_{2}(L)\cong L/[L,x]/\mathrm{Z}(L/[L,x])$$
is isomorphic to $H(1,0)$, since $L/\mathrm{Z}_{2}(L)$ is not abelian. Hence $\mathrm{sdim}L/\mathrm{Z}(L)=(4,0)$, contradicting the assumption that $\mathrm{sdim}(L/\mathrm{Z}(L))^{2}=(1, 1)$.

\end{proof}

\begin{definition}
A Lie superalgebra $L$ is called capable if  there is a Lie superalgebra $H$ such that $L\cong H/\mathrm{Z}(H)$.
\end{definition}
As in Lie algebra case \cite[Theorem 21]{K.Moneyhun}, if $L$ is capable and $(K,M)$ is a maximal defining pair of $L$, then $M=\mathrm{Z}(K).$
\begin{lemma}\label{10}
Let $L$ be a non-capable, nilpotent, non-abelian Lie superalgebra of superdimension $(s, t).$ Then $(s-1,t)<\mathrm{smr}(L)$ or $(t,s)<\mathrm{smr}(L).$
\end{lemma}

\begin{proof}
Let $(K,M)$ be a maximal defining pair of $L$. Since $L$ is not abeian, we have $L\cong K/M$ and $M\subsetneq K^{2}$. Since $L$ is not capable, we have $M\subsetneq \mathrm{Z}(K)$ and $\mathrm{sdim}K/\mathrm{Z}(K)\leq (s-1, t)$ or $(s, t-1).$
Since $\mathcal M\subsetneq K^{2}$,   by Lemma \ref{lemma1}, one may easily obtain that
 $\mathrm{smr}(L)>(s-1, t)$ or $\mathrm{smr}(L)>(t, s).$
\end{proof}

\begin{lemma}\label{11}
Let $L$ be a capable, nilpotent, non-abelian Lie superalgebra.
Then $\mathrm{sdr}(\mathcal{C}(L))<\mathrm{smr}(L).$
\end{lemma}

\begin{proof}Let $(K,M)$ be a maximal defining pair of $L$. We have
$L\cong K/M$ and $M\subset \mathrm{Z}(K)\cap K^{2}.$ Since $L$ is capable, we have $M=\mathrm{Z}(K)$. Since $L$ is not abeian, we have $M\subsetneq K^{2}$.  It follows that  $\mathrm{sdr}(\mathcal{C}(L))<\mathrm{smr}(L).$
\end{proof}

\begin{proposition}\label{ww}
Let $L$ be a finite-dimentional, non-abelian, nilpotent Lie superalgebra. Then
\begin{itemize}
\item[(1)]$\mathrm{smr}(L)\neq(0,2).$
\item[(2)]$\mathrm{smr}L=(2,0)$ if and only if $L\cong H(1,0)\oplus \mathrm{Ab}(1,0)$.
\item[(3)]$\mathrm{smr}(L)=(1,1)$ if and only if $L$ is isomorphic to one of the following Lie superalgebras:
\begin{itemize}
\item[(3.1)] $H(1, 0)\oplus \mathrm{Ab}(0,1)$;
\item[(3.2)] $H(0,1)$.
\end{itemize}
\end{itemize}
\end{proposition}

\begin{proof}Suppose $\mathrm{sdim}L=(s,t)$.

(1) Assume conversely that $\mathrm{smr}(L)=(0,2).$ Then by Proposition \ref{proposition3}, $L$ is not abelian. First suppose $L$ is  not capable. Then by Lemma \ref{10}, we have $(s,t)<\mathrm{smr}(L)+(1, 0)$ or $(t, s)<\mathrm{smr}(L)$. Since $L$ is nilpotent and not abeian, we must have $\mathrm{sdim}L=(1,1).$ It is easy to deduce that
$L\cong H(0,1)$. Consequently, $\mathrm{smr}(L)=(1,1),$ contradicting the assumption that $\mathrm{smr}(L)=(0,2).$

Next suppose $L$ is capable and $K$ is a cover of $L$.  Then we have $L\cong K/\mathrm{Z}(K).$ By Lemma \ref{11}, we have $\mathrm{sdr}(K)<\mathrm{smr}(L)=(0,2)$. If $\mathrm{sdr}(K)=(0,0)$, then by Lemma \ref{lemma12}, $L$ is either abelian, a contradiction, or $L=H(1,0)$, which yields $\mathrm{smr}(L)=(1,0),$ also a contradiction. Hence $\mathrm{sdr}(K)=(0,1).$ Therefore,  $L\cong K/\mathrm{Z}(K)$ is one of the Lie superalgebras listed in Lemma \ref{balabalala}. Then  $\mathrm{smr}(L)=(0,0)$, $(1,1)$ or $(1,1)$, which is impossible. Hence, $\mathrm{smr}(L)\neq(0,2).$

(2) Suppose $L=H(1, 0)\oplus \mathrm{Ab}(1,0)$. Then by Propositions \ref{proposition3}, \ref{hsda}  and Lemma \ref{aaa}, one may compute
$\mathrm{sdim}\mathcal{M}(L)=(4, 0).$
Then, since $\mathrm{sdim}L=(4,0)$, we have $\mathrm{smr}(L)=(2, 0).$

Conversely, suppose $\mathrm{smr}(L)=(2,0).$ By Proposition \ref{proposition3}, $L$ is not abelian. First suppose $L$ is not capable. Then by Lemma \ref{10}, we have $(s,t)<\mathrm{smr}(L)+(1, 0)$ or $(t, s)<\mathrm{smr}(L)$, contradicting the assumption that $L$ is nilpotent and not abeian.

Next suppose $L$ is capable and $K$ is a cover of $L$. Then we have $L\cong K/\mathrm{Z}(K).$
By Lemma \ref{11}, we have ${\mathrm{sdr}}(K)<\mathrm{smr}(L)=(2,0)$. If $\mathrm{sdr}(K)=(0,0),$ then by Lemma \ref{lemma12}, either $L$ is abelian, a contradiction, or $L=H(1, 0)$, which yields $\mathrm{smr}(L)=(1, 0),$ also a contradiction. Hence $\mathrm{sdr}(K)=(1, 0).$ Therefore, $L\cong K/\mathrm{Z}(K)$ is one of the superalgebras listed in Lemma \ref{balabala}. A direct verification shows that $L\cong H(1,0)\oplus \mathrm{Ab}(1,0).$

(3) Suppose $L=H(1,0)\oplus \mathrm{Ab}(0,1).$ By Propositions \ref{hsda}, \ref{proposition3} and Lemma \ref{aaa}, one may compute $\mathrm{sdim}\mathcal{M}(L)=(3,2).$
Then, since $(s,t)=(3,1),$ we have $\mathrm{smr}(L)=(1,1).$

Suppose $L=H(0,1).$ Then by Proposition \ref{hsda}, we have $\mathrm{sdim}\mathcal{M}(L)=(0,0).$ Then since $(s,t)=(1,1),$ we have $\mathrm{smr}(L)=(1,1).$

Conversely, suppose $\mathrm{smr}(L)=(1,1).$ By Proposition \ref{proposition3}, $L$ is not abelian. First suppose $L$ is not capable. Then by Lemma \ref{10}, we have $(s,t)<\mathrm{smr}(L)+(1, 0)$ or $(t,s)<\mathrm{smr}(L)$. Since $L$ is nilpotent and not abeian, we must have $\mathrm{sdim}L=(1,1).$ Then it is easy to deduce that
$L\cong H(0,1)$.

Next suppose $L$ is capable and $K$ is a cover of $L$. Then we have $L\cong K/\mathrm{Z}(K).$
By Lemma \ref{11}, we have ${\mathrm{sdr}}(K)<\mathrm{smr}(L)=(1,1)$. If $\mathrm{sdr}(K)=(0,0),$ then either $L$ is abelian, a contradiction, or $L=H(1, 0)$, which yields $\mathrm{smr}(L)=(1,0),$ also a contradiction. If $\mathrm{sdr}(K)=(1,0),$ then $L\cong K/\mathrm{Z}(K)$ is one of the Lie superalgebras listed in Lemma \ref{balabala}, and then $\mathrm{smr}(L)=(0,0), (1,0), (2,0) $ or $(4,0),$ a contradiction.  Suppose $\mathrm{sdr}(K)=(0,1).$ Then $L\cong K/\mathrm{Z}(K)$ is one of the superalgebras listed in Lemma \ref{balabalala}. A direct verification shows that $L\cong H(1,0)\oplus \mathrm{Ab}(0,1)$ or $H(0,1)$.
\end{proof}

We assemble Propositions \ref{proposition3}, \ref{13}, \ref{ww} to be the following classification theorem. Recall that $\mathrm{Ab}(m,n)$ denotes  the abelian Lie superalgebra of superdimension $(m,n)$ and $H(m, n)$ is the $(2m+1,n)$-dimensional Heisenberg Lie superalgebras of even center.
\begin{theorem}
Up to isomorphism, all the finite-dimensional nilpotent Lie superalgebras $L$ of  multiplier-rank $\leq2$ are listed below:
$$
\begin{array}{c|c}
    \mathrm{smr}(L)  & L    \\\hline
      (0,0) & \mathrm{any}\ \mathrm{abelian}\ \mathrm{Lie}\ \mathrm{superalgebra}     \\ \hline
      (1,0)   & H(1,0)    \\\hline
       (2,0)  &  H(1,0)\oplus \mathrm{Ab}(1,0)     \\\hline
       (1,1)   &  H(1,0)\oplus \mathrm{ Ab}(0,1)    \\\hline
    (1,1)&  H(0,1)   \\
\end{array}
$$
\qed
\end{theorem}

\end{document}